\documentclass[a4paper,12pt]{article}
\usepackage[T1]{fontenc}
\usepackage{graphicx}

\usepackage[a4paper,margin=1in]{geometry}
\usepackage{indentfirst}

\usepackage{amsmath,amsthm,amsfonts,amssymb,bm,wasysym}
\usepackage[mathscr]{euscript}
\usepackage[usenames]{color}
\usepackage{hyperref}

\DeclareFontFamily{U}{BOONDOX-calo}{\skewchar\font=45 }
\DeclareFontShape{U}{BOONDOX-calo}{m}{n}{
  <-> s*[1.05] BOONDOX-r-calo}{}
\DeclareFontShape{U}{BOONDOX-calo}{b}{n}{
  <-> s*[1.05] BOONDOX-b-calo}{}
\DeclareMathAlphabet{\mathcalboondox}{U}{BOONDOX-calo}{m}{n}
\SetMathAlphabet{\mathcalboondox}{bold}{U}{BOONDOX-calo}{b}{n}
\DeclareMathAlphabet{\mathbcalboondox}{U}{BOONDOX-calo}{b}{n}
\newcommand{\Y}{\mathcalboondox{Y}}

\numberwithin{equation}{section}

\theoremstyle{definition}
\newtheorem{teo}{Theorem}[section]
\newtheorem{conjecture}{Conjecture}[section]
\newtheorem{lema}{Lemma}[section]
\newtheorem{prop}{Proposition}[section]
\newtheorem{deft}{Definition}[section]
\newtheorem{cor}{Corollary}[section]

\newtheorem*{rmk}{Remark}

\DeclareMathOperator\Area{Area}

\DeclareMathOperator\diff{d}
\DeclareMathOperator*{\esssup}{ess\,sup}

\newcommand{\C}{\mathbb{C}}
\newcommand{\D}{\mathbb{D}}
\newcommand{\h}{\mathbb{H}}

\newcommand{\Q}{\mathbb{Q}}
\newcommand{\R}{\mathbb{R}}
\newcommand{\s}{\mathbb{S}}
\newcommand{\Z}{\mathbb{Z}}

\allowdisplaybreaks

\title{David regularity of the Yoccoz extension}
\author{Luna Lomonaco, Lucas Oliveira, Miguel Ratis Laude}

\begin{document}

\maketitle
\begin{abstract}
    A central problem in the study of critical circle dynamics is understanding the regularity of Yoccoz conjugators—circle homeomorphisms that conjugate critical circle maps with irrational rotation numbers to their corresponding rigid rotations. One can approach this problem from a different angle by studying the regularity of extensions of these maps to the unit disk. Of particular interest is the question of when such a conjugator admits a David extension. Building on the work of Petersen and Zakeri, we classify the David regularity of a specific extension process known as the Yoccoz extension.
\end{abstract}

\section{Introduction}

Let $f:\s^{1}\to\s^{1}$ be an orientation preserving circle homeomorphism.
Its rotation number  $\rho(f)$ is defined as  
$$\rho(f)= \lim_{n \rightarrow \infty} \frac{F^n(x)-x}{n},$$
where $F:\R\to\R $ is a lift of $f$: an increasing homeomorphism of the real line such that for all $x \in \R,\,\pi \circ F(x)=f\circ \pi(x)$, where $\pi(x)=e^{2\pi i x}$.
Denote by $R_\theta: \s^1 \rightarrow \s^1$ the rigid rotation by $\theta$.
Denjoy proved in \cite{Denjoy1932} that, if $f:\s^{1}\to\s^{1}$ is an orientation preserving $C^1$-diffeomerphism with derivative of bounded variation, and with $\rho(f)\in \R\setminus \Q$, then there exists a  homeomorphism $h:\s^{1}\to\s^{1}$ conjugating $f$ with $R_{\rho(f)}$.
In \cite{Yoccoz1984} Yoccoz extended Denjoy's result to real analytic critical circle maps: real analytic maps $f:\s^{1}\to\s^{1}$ with at least one critical point, this is with at least a point $\omega\in\s^{1}$ such that $f'(\omega)=0$. More precisely, Yoccoz proved 
 that, if $f:\s^{1}\to\s^{1}$ is a 
real analytic critical circle map with $\rho(f)\in \R\setminus \Q$, then there exists a  homeomorphism $h:\s^{1}\to\s^{1}$ conjugating $f$ with $R_{\rho(f)}$.
In \cite{Herman} Herman improved the regularity of the conjugating homeomorphism for critical real analytic circle maps with rotation number satisfying some arithmetic property. Specifically, Herman proved that,  if $f$ is  real analytic critical circle map with $\rho(f)\in \R\setminus \Q$, then
\begin{equation}\label{intro1}
    \rho(f)\in BT \iff f=h^{-1}\circ R_{\rho(f)}\circ h \mbox{ with }h\in QS(\s^{1}) \,,
\end{equation}
with $\rho(f)\in BT$ when, denoting as $[a_0,a_1,a_2,\dots]$ the continued fraction expansion of $\rho(f)$, we have 
$$
\sup a_{n}<\infty\,,
$$
and with $h\in QS(\s^{1})$ if $h$ is a \emph{quasisymmetric homeomorphism}: an orientation preserving homeomorphism of
$\s^{1}$
for which there exists a $K\geq 1$ such that, for all $x \in\R$ and $t>0$, we have 
$$\frac{1}{K}\leq \frac{H(x+t)-H(x)}{H(x)-H(x-t)}\leq K\,,$$
where $H:\R\to\R$ is a lift of $h$. In \cite{Swiatek1998}, \'{S}wi\c{a}tek extends Herman's result to circle homeomorphisms with a geometric control of some iterated cross ratios and a notion of lateral criticality.\\

Quasiymmetric homeomorphisms can be regarded as the boundary values of quasiconformal mappings (see \cite{Beurling-Ahlfors1956}). An orientation preserving homeomorphism $\phi:U \subseteq \C \rightarrow V \subseteq \C$ is said to be quasiconformal if  $\phi$ is in the Sobolev space $W^{1,2}(U)$ of functions whose first-order distributional derivatives are in $L^2(U)$, and there exists $k<1$ such that, setting $\mu_{\phi}(z):=\frac{\overline{\partial}\phi(z)}{\partial\phi(z)}$, we have $|\mu_{\phi}(z)|\leq k$ a.e. We call $\mu_{\phi}(z)$ the Beltrami coefficient of $\phi$ at $z$, 
$$K_{\phi}(z)=\frac{|\mu_{\phi}(z)|+1}{|\mu_{\phi}(z)|-1}$$
the dilatation (or conformal distortion) of $\phi$ at $z$, and $$K_{\phi}=\mbox{ess sup}_{z \in U}K_{\phi}(z)$$ the dilatation of $\phi$. Note that, when $\phi$ is quasiconformal, $K_{\phi}<\infty$.
The set of quasiconformal maps $\phi:U\rightarrow U$ form a group under composition: if $\phi$ is quasiconformal, its inverse is also quasiconformal, with $K_{\phi}=K_{\phi^{-1}}$, and if $\phi$ and $\psi$ are quasiconformal, $\phi\circ\psi$ is quasiconformal with $K_{\phi\circ\psi}\leq K_{\phi}K_{\psi}$.

In \cite{David1988}, David introduced a generalization of quasiconformal maps, the so-called David homeomorphisms, where the dilatation is allowed to explode, but in a controlled way. More precisely, an orientation preserving homeomorphism $\phi:U \subseteq \C \rightarrow V \subseteq \C$ is said to be a David map if  $\phi$ is in the Sobolev space $W^{1,2}(U)$, and there exist constants $\alpha>0$ and $C>0$ such that 
 \begin{equation}
    \mbox{A}_\sigma\left\{z\in U:  K_{\phi}(z) > K\right\}\le C e^{-\alpha K},
    \end{equation}
    where $\mbox{A}_\sigma$ is the spherical area:
    $$\mbox{A}_\sigma(B)=\iint_B \frac{\mbox{dx dy}}{(1+|z|^2)^2}. $$
The set of David maps $\phi:U\rightarrow U$ do not form a group under composition: in particular, the inverse of a David map is not necessarily David. David maps with David inverse are called \textit{biDavid maps} in what follows.  It is important to emphasize that the boundary behavior of David maps is not known: there exist sufficient conditions for a circle homeomorphism to be the boundary value of a David map, and there exist necessary conditions, but the gap between the necessary and the sufficient conditions is definitely not negligible (see \cite{Zakeri2008}).\\

In \cite{Petersen-Zakeri2004} Petersen and Zakeri proved that, if $f$ is  real analytic critical circle map with $\rho(f)\in \R\setminus \Q$ and, denoting as $[a_0,a_1,a_2,\dots]$ the continued fraction expansion of $\rho(f)$, we have  $$\log a_{n}=\mathcal{O}(\sqrt{n}) \mbox{ as } n\to\infty,$$ then the homeomorphism $h:\s^{1}\to\s^{1}$ conjugating $f$ with $R_{\rho(f)}$ admits a David extension $H : \overline{\D}\to \overline{\D}$.
The extension they constructed is the so-called Yoccoz extension $\Y(h):\overline{\h}\to \overline{\h}$, an extension obtained iteratively and based on the dynamical partitions of $f$ and $R_{\rho(f)}$.
They conjectured that the converse should be true, that is, if the normalized linearizing map $h:\s^1\to \s^1$, which satisfies $h \circ f = R_{\rho(f)} \circ h$, admits a David extension $H : \overline{\D}\to \overline{\D}$, then the continuous fraction expansion of $\rho(f)$ satisfies $\log a_{n}=\mathcal{O}(\sqrt{n}) \mbox{ as } n\to\infty$. We prove the conjecture for the extension process they used.\\

More precisely, let $\theta \in \R\setminus \Q$, and $[a_0,a_1,a_2,\dots]$ its continued fraction expansion. Define 
$$BT:=\{\theta \in \R\setminus \Q \,\,|\,\,\sup a_{n}<\infty\,\},$$ 
$$PZ:=\{\theta \in \R\setminus \Q\,\,|\,\,\log a_{n}=\mathcal{O}(\sqrt{n}) \mbox{ as } n\to\infty\},$$
$$\mathcal{A}:=\{ \theta \in \R\setminus \Q\,\,|\,\,\log^{2}a_{n+1}\lesssim \sum_{k=1}^{n}\log(a_{k} +1)\,\}.$$

Note that $BT \subset PZ \subset \mathcal{A}$, and all the inclusions are strict.
It follows from the work of Herman and \'{S}wi\c{a}tek that
$$\rho(f) \in BT \mbox{ if, and only if, } \Y(h) \mbox { is a quasiconformal mapping}.$$
Here, we prove the following:
\begin{teo}\label{teo.main}
Let $f:\s^{1}\to\s^{1}$ be a critical circle map with $\rho(f)\in \R\setminus \Q$, and let $F:\R\to \R$ be a lift of it. Let $h$ denote the lift of the normalized linearizing map $h:\R\to \R$, which satisfies $h \circ F = \hat R_{\rho(f)} \circ h$, where $\hat R_{\rho(f)}(z)=z+\rho(f)$. Let $\Y(h):\overline{\h}\to \overline{\h}$ be the Yoccoz extension of $h$. Then:
\begin{itemize}
    \item [a]. $\rho(f) \in PZ$ if, and only if, $\Y(h)$ is a David mapping;
    \item [b]. $\rho(f) \in \mathcal{A}$ if, and only if, $\Y(h)^{-1}$ is a David mapping.
\end{itemize}
\end{teo}
We obtain as a corollary that the David extension $H : \overline{\D}\to \overline{\D}$ of the homeomorphism $h:\s^{1}\to\s^{1}$ conjugating $f$ with $R_{\rho(f)}$ in Petersen and Zakeri Theorem (see \cite{Petersen-Zakeri2004}) is actually biDavid (see Corollary \ref{cor}).\\

We then generalize the result to strong David maps and maps of finite distortion (see Section \ref{SG} for definitions). More precisely, we find the necessary and sufficient arithmetic condition for $\Y(h)$ and $\Y(h)^{-1}$ to be strong David or of finite distortion (see Corollary \ref{SD} and Corollary \ref{FD}). We end the section with a lower bound on the quasiconformal dilatation of the piece-wise affine approximations of the conjugator at level $n$, which allows us to compare the Yoccoz extension with other classical ones (see Proposition \ref{qs}).\\

\textbf{Acknowledgements.}
The first author was supported partially by the Serrapilheira Institute (grant number Serra-1811-26166), the FAPERJ - Fundação Carlos Chagas Filho de Amparo à Pesquisa do Estado do Rio de Janeiro (grant number JCNE - E26/201.279/2022 and JCM - E-26/210.016/2024), the ICTP through the Associates Programme and from the Simons Foundation through grant number 284558FY19. The second author was partially supported by the Office of Naval Research through grant GRANT14201749 (award number N629092412126). The third author was supported by Capes.

The authors would also like to thank Pablo Guarino and João Pedro Ramos for fruitful conversations that lead to improvements of the results presented here.

\section{Preliminaries: the Yoccoz Extension}
In this section, following Petersen and Zakeri, we will carefully introduce the Yoccoz extension. The concept behind it is really natural and, as we will see, the extension carries all the information needed to understand the arithmetic properties of the rotation number and the regularity of the conjugator (this last property in an intrinsic way).

In what follows, let $F: \s^1 \to \s^1$ be an analytic critical homeomorphism of the circle, normalized to have its critical point at $1$. We assume the rotation number $\theta$ of $F$ is irrational. By a theorem of Yoccoz \cite{Yoccoz1984}, there exists a homeomorphism $H:\s^1\to \s^1$ that conjugates $F$ with the rigid rotation $R_\theta$ by $\theta$; we shall call it the \textit{Yoccoz conjugator} or the \textit{linearizing map}. As $\theta$ is irrational, it admits a representation in continued fractions:
\[ \theta = \frac{1}{a_1 + \frac{1}{a_2 + \dots}} =: [a_1, a_2, \dots] \]
and we shall denote by $p_n/q_n$ (in lowest terms) the approximant
\[ \frac{p_n}{q_n} := [a_1, \dots, a_n] = \frac{1}{a_1 + \frac{1}{\dots + \frac{1}{a_n}}}. \]
This means that the terms $q_n$ are the \textit{closest return times} for $F$: given any $x \in \s^1$, the point $F^{-q_n}(x)$ is the closest to $x$ out of all points $F^{-m}(x)$, with $m < q_{n+1}$.

It will be simpler throughout most of the text to work with a \textit{lift} of $F$ to the real line --- so a map $f:\R \to \R$ such that $f(x) = F(e^{2\pi ix})$. Thus $f$ is an analytic homeomorphism of the real line, satisfying $f(x + 1) = f(x) + 1$, with a critical point at $0$ (indeed, at any integer). The Yoccoz conjugator also lifts to a homeomorphism $h:\R \to \R$ conjugating $f$ with translation $T_\theta$ by $\theta$.

As to fix notation, given two sequences $(A_n)_{n\geq 1}$ and $(B_n)_{n\geq 1}$, we will say they are \textit{comparable}, and denote $A_n \asymp B_n$, when there exist index $n_0 \geq 1$ and constant $C \geq 1$, not depending on $n$, for which
\[ \frac{1}{C}A_n \leq B_n \leq C A_n, \ \forall n \geq n_0. \]
If we only have that $B_n \leq C A_n$ for all $n \geq n_0$, we will denote $B_n \lesssim A_n$ (this is the same thing as writing $B_n = \mathcal{O}(A_n)$). More generally, we might be dealing with collections of values $(\mathcal{A}_n)_{n\geq 1}$ and $(\mathcal{B}_n)_{n\geq 1}$, instead of a single sequence (e.g. the lengths of intervals of finer and finer partitions); in those situations, we will denote $A \asymp B$ (and similarly $A \lesssim B$) if the comparison constant is independent of $n$ and of the specific elements $A \in \mathcal{A}_n, B \in \mathcal{B}_n$ we pick.

\subsection{Modified dynamical partitions}

The main idea for obtaining the Yoccoz extension of $h$ is to obtain \textit{dynamically defined grids} on $\h$, with good geometric control for each of their cells, and attempt an extension at each cell. For that, we start by considering the following dynamical partition:
\[ \mathcal{Q}_n := \{ f^{-j}(0) \ | \ 0\leq j < q_n \} + \Z. \]
The partition $\mathcal{Q}_n$ projects to the circle, partitioning it like
\[ \s^1 = \left( \bigcup_{j = 0}^{q_n - q_{n-1} - 1}[F^{-j-q_{n-1}}(1), F^{-j}(1)] \right) \cup \left( \bigcup_{j=0}^{q_{n-1} - 1}[F^{-j}(1), F^{-j-q_n+q_{n-1}}(1)] \right). \]
Indeed, the intervals of the form $[F^{-j-q_{n-1}}(1), F^{-j}(1)]$ are the closest return intervals of time $q_{n-1}$, meaning no other point of the partition lies in between them, while the intervals of the form $[F^{-j}(1), F^{-j-q_n+q_{n-1}}(1)]$ are the unions of two intervals: $[F^{-j}(1), F^{-j-q_n}(1)]$, a closest return interval of time $q_n$; and $[F^{-j-q_n}(1), F^{-j-q_n+q_{n-1}}(1)]$, a closest return interval of time $q_{n-1}$. Since the middle point $F^{-j-q_n}(1)$ is not in the partition, the other two points are the closest to each other in $\mathcal{Q}_n$. One can now verify that, when refining into $\mathcal{Q}_{n+1}$, the intervals above are partitioned into new ones at the points
\[ F^{-j-q_{n-1}}(1), F^{-j-q_{n-1}-q_n}(1), \dots, F^{-j-q_{n-1}-(a_{n+1}-1)q_n}(1) = F^{-j-q_{n+1}+q_n}(1), F^{-j}(1) \]
in the first case and
\[ F^{-j}(1), F^{-j-q_n}(1), \dots, F^{-j-a_{n+1}q_n}(1) = F^{-j-q_{n+1}+q_n}(1), F^{-j-q_n+q_{n-1}}(1) \]
in the second case; see Figure \ref{fig:dynamical-partition}. Notice that they are laid in that order by being closest return times. This shows us that, if one takes $t, s \in \mathcal{Q}_n$ two adjacent points and write
\[ [t, s]\cap \mathcal{Q}_{n+1} := \{ t = t_0 < t_1 < \dots < t_k = s \} \]
then we have $k = a_{n+1}$ or $a_{n+1} + 1$, depending on the type of interval.

We recall that the usual dynamical partition $\mathcal{P}_n$ is given by
\[ \mathcal{P}_n := \{ f^{-j}(0) \ | \ 0 \leq j < q_n + q_{n-1} \} + \Z. \]
The intervals defined by $\mathcal{P}_n$ also come in two forms (again projecting down to $\s^1$): $[F^{-j-q_{n-1}}(1), F^{-j}(1)]$, for $0 \leq j < q_n$, and $[F^{-j}(1), F^{-j-q_n}(1)]$, for $0 \leq j < q_{n-1}$ (see \cite{deFaria-Guarino}). Notice that the first $q_n - q_{n-1}$ intervals of $\mathcal{Q}_n$ --- the ones of the form $[F^{-j-q_{n-1}}(1), F^{-j}(1)]$ --- are included in the first type of interval of $\mathcal{P}_n$; the other $q_{n-1}$ intervals of $\mathcal{Q}_n$ are unions of two intervals of $\mathcal{P}_n$. Indeed, the remaining intervals of first type have indices $q_n - q_{n-1}\leq j < q_n$, and so can be rewritten as
\[ [F^{-j-q_n+q_{n-1}}(1), F^{-j-q_n}(1)], \ 0 \leq j < q_{n-1}, \]
which, when joined with the $q_{n-1}$ intervals of the second type, will form the remaining intervals of $\mathcal{Q}_n$. This means that we can transfer certain properties of $\mathcal{P}_n$ to $\mathcal{Q}_n$, for instance the following \textit{a priori} bounds (originally from \cite{Herman1987, Swiatek1988}, see also Theorem 3.1 and its proof in \cite{deFaria-deMelo1999}):

\begin{figure}
    \centering
    \includegraphics[width=0.475\linewidth]{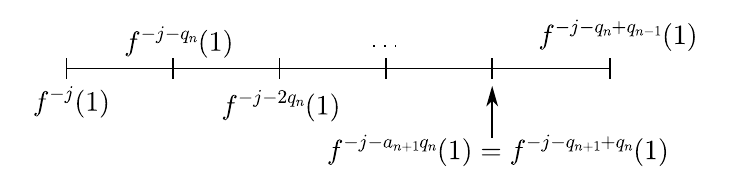}
    \hspace{0.5cm}
    \includegraphics[width=0.475\linewidth]{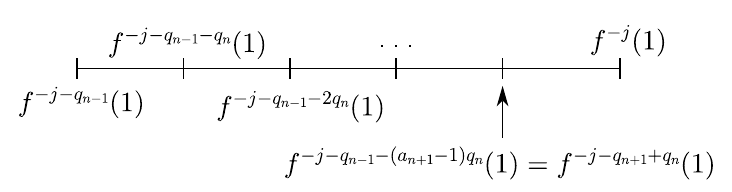}
    \caption{The two different kinds of intervals in the partiton $\mathcal{Q}_n$. The intermediate points are closest returns of time $q_n$.}
    \label{fig:dynamical-partition}
\end{figure}

\begin{teo}[Herman-\'{S}wi\c{a}tek; Real \textit{a priori} bounds]\label{teo.aprioribounds}
    
    Let $f$ be a $\mathcal{C}^r$ critical circle map, with $r \geq 3$, and $\mathcal{P}_n$ be the associated dynamical partition. Then the followiong hold:
    \begin{itemize}
        \item if $I, J$ are adjacent intervals determined by $\mathcal{P}_n$, then $|I| \asymp |J|$ (meaning the comparison constant does not depend on $n$ nor the intervals);
        \item if $I_n := [0, f^{-q_n}(0) + p_n]$, then $|I_n| \asymp |I_{n+1}|$.
    \end{itemize}
    
\end{teo}
\begin{cor}\label{cor.aprioribounds}
    
    There are constants $0 < \varepsilon \leq \sigma < 1$ such that, for any interval $I$ of $\mathcal{P}_n$,
    \[ \varepsilon^n \lesssim |I_n| \leq |I| \lesssim \sigma^n. \]
    
\end{cor}

Notice that $I_n$ is the lift of the closest return interval $[1, F^{-q_n}(1)]$ of time $q_n$, and therefore we indeed have $|I_n| \leq |I|$ for all intervals $I$ determined by $\mathcal{P}_n$. These results transfer trivially to $\mathcal{Q}_n$ from the observation that the intervals determined by it are either also determined by $\mathcal{P}_n$ or the union of two adjacent intervals determined by $\mathcal{P}_n$.

\subsection{Grids, cells and the grid map}
We turn to defining the grids now. For each $t \in \mathcal{Q}_n$, let the \textit{height} at level $n$ be $y_n(t) := (t_l - t_r)/2$, with $t_l \in \mathcal{Q}_n$ the point to the left of $t$ and $t_r \in \mathcal{Q}_n$ the point to the right --- i.e. $y_n(t)$ is the average between the two adjacent intervals of the partition at level $n$ that share $t$ as a boundary point. The points $z_n(t) := t + iy_n(t)$ for $t \in \mathcal{Q}_n, n \geq 0$ will then form a grid $\Gamma$ when we connect:
\begin{itemize}
    \item $z_n(t)$ to $z_{n+1}(t)$ with a vertical segment, for each $t \in \mathcal{Q}_n$;
    \item $z_n(t)$ to $z_n(s)$ with a line segment whenever $t, s \in \mathcal{Q}_n$ are adjacent.
\end{itemize}

We may repeat the previous steps, now replacing the function $f$ with the translation $T_\theta$ by $\theta$, to obtain partitions $\mathcal{Q}'_n$, height functions $y'_n$, vertices $z'_n$, and a grid $\Gamma'$. It is clear that $\Gamma$ and $\Gamma'$ are homeomorphic by simply sending $z_n(t)$ to $z_n'(h(t))$ and affinely mapping corresponding segments; see Figure \ref{fig:grid-map}.

\begin{figure}[h!]
    \centering
    \includegraphics[scale = 0.26]{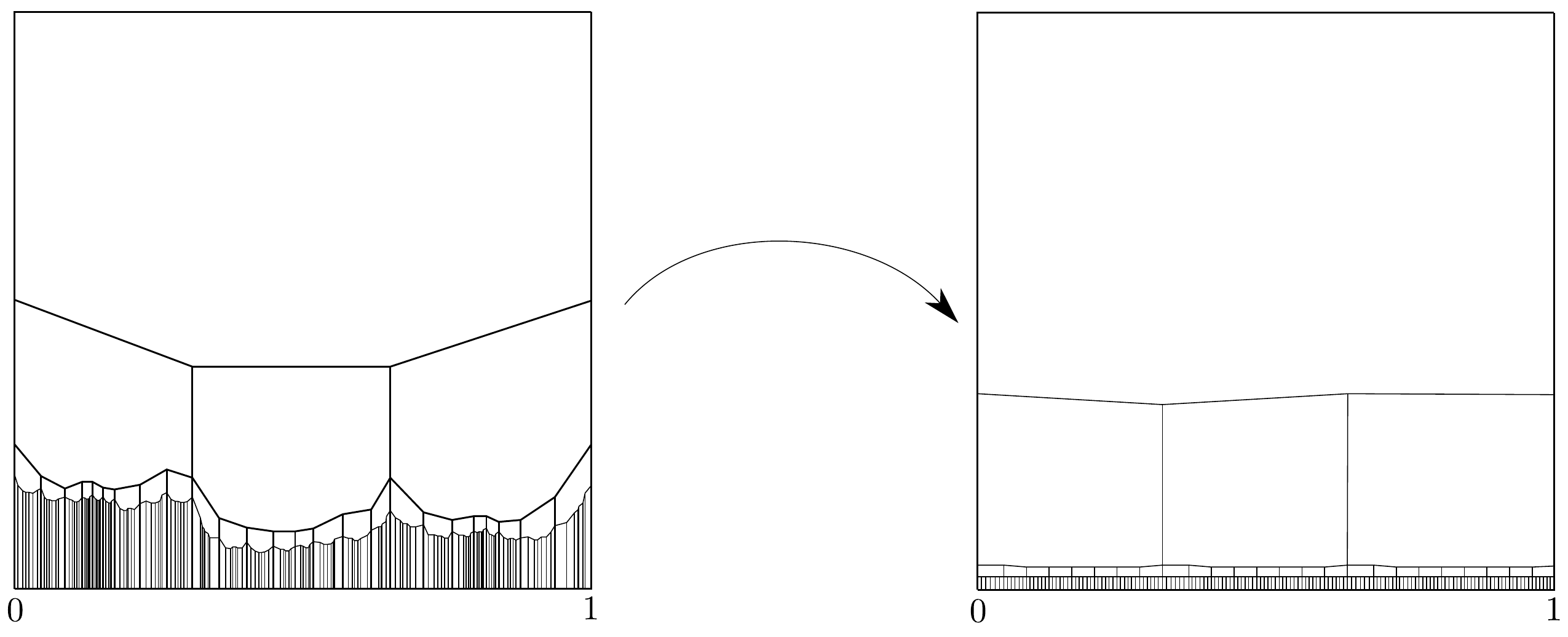}
    \caption{The base grids $\Gamma$ (left) and $\Gamma'$ (right). Each vertex is a point of the form $z_n(x)$ or $z'_n(x)$. Only three levels have been drawn.}
    \label{fig:grid-map}
\end{figure}

The Yoccoz extension will be constructed so as to coincide with this homeomorphism on the grid $\Gamma$. Notice that the partition $\mathcal{Q}'_n$ still satisfies the first item of Theorem \ref{teo.aprioribounds} --- any two adjacent intervals are comparable. In fact, any two intervals of the same level are comparable, independent of being adjacent. This is simply because the intervals have either length $|q_{n-1}\theta - p_{n-1}|$ (a closest return of time $q_{n-1}$) or $|q_n\theta - p_n| + |q_{n-1}\theta - p_{n-1}|$ (the union of a closest return of time $q_n$ and another of time $q_{n-1}$), and we see that
\[ |q_{n-1}\theta - p_{n-1}| \leq |q_n\theta - p_n| + |q_{n-1}\theta - p_{n-1}| \leq 2|q_{n-1}\theta - p_{n-1}|. \]

We first observe that these \textit{a priori} bounds imply that the cells of the grids $\Gamma, \Gamma'$ have controlled geometry.
\begin{lema}\label{lema.geometric_control}

    Let $Q$ be a cell of the grid $\Gamma$ or of the grid $\Gamma'$. Then:
    \begin{itemize}
        \item the top, left and right sides of $Q$ are all comparable (with constant not depending on the level or the cell);
        \item the inner angle between the top and the left/right side of the polygon $Q$ is uniformly bounded away from $0$ and $\pi$;
        \item the inner angles between any two adjacent sides in the bottom of $Q$ are uniformly bounded away from $0$ and $2\pi$.
    \end{itemize}
    
\end{lema}

\begin{figure}[h]
    \centering
    \includegraphics[scale = 0.7]{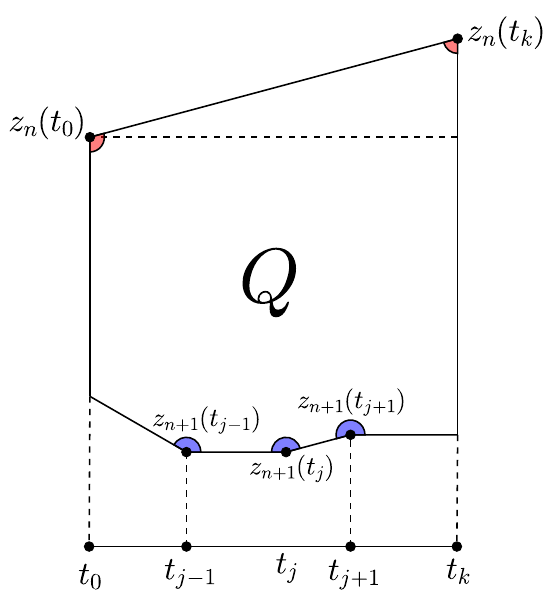}
    \caption{A possible cell $Q$. In red, the angles between the top side and the left and right sides. In blue, the angles between adjacent bottom sides.}
    \label{fig:cell-geometric-control}
\end{figure}
 
See Figure \ref{fig:cell-geometric-control}. We will work only on $\Gamma$, as the proof for $\Gamma'$ is completely analogous.

\begin{proof}
    Let $I = [t_0, t_k]$ be the interval determined by $\mathcal{Q}_n$ associated to $Q$, with $[t_0, t_k]\cap \mathcal{Q}_{n+1} = \{ t_0 < \dots < t_j < \dots < t_k \}$. We will begin by showing that the lengths of the left, right and top sides of $Q$ are all comparable to $|I|$. Let $I_l$ be the interval of $\mathcal{Q}_n$ to the left of $I$, $J$ be the leftmost subinterval of $I$ in $\mathcal{Q}_{n+1}$ and $J_l$ be the interval of $\mathcal{Q}_{n+1}$ to the left of $J$ --- i.e. the rightmost subinterval of $I_l$. Then the length of the left side of $Q$, which is the vertical segment connecting $z_n(t_0)$ to $z_{n+1}(t_0)$ (see Figure \ref{fig:cell-geometric-control}), is given by
    \[ |z_n(t_0) - z_{n+1}(t_0)| = y_n(t_0) - y_{n+1}(t_0) = \frac{|I| + |I_l|}{2} - \frac{|J| + |J_l|}{2} = \frac{1}{2}[ (|I| - |J|) + (|I_l| - |J_l|) ], \]
    where the second equality just follows from the definition of the $y_n$. To see that this is comparable to $|I|$, it is enough to show that $|I| - |J| \asymp |I|$ (the same argument shows that $|I_l| - |J_l| \asymp |I_l|$, and by Theorem \ref{teo.aprioribounds} we have $|I_l| \asymp |I|$). On one side, we trivially have $|I| - |J| \leq |I|$. On the other side, since any adjacent interval to $J$ of level $n + 1$ has length comparable to $|J|$, and as $J$ is a subinterval of $I$, we must have $|I| \geq (1 + 1/C)|J|$ for some constant $C \geq 1$ not depending on $I, J, Q$ or $n$. Easy manipulations now show that, $|I| - |J| \geq |I|/(C + 1)$.

    The claim for the top side comes from similar computations: if we denote $I_r$ the interval determined by $\mathcal{Q}_n$ to the right of $I$, we then see that
    \[ z_n(t_0) = t_0 + i\frac{|I| + |I_l|}{2} \ \text{ and } \ z_n(t_k) = t_k + i\frac{|I| + |I_r|}{2}, \]
    meaning that the length of the top side is given by
    \[ |z_n(t_0) - z_n(t_k)| = \left| t_k - t_0 + i\frac{|I_l| - |I_r|}{2} \right|. \]
    On one side, this is greater than or equal to $|t_k - t_0| = |I|$; on the other, we have
    \[ \left| t_k - t_0 + i\frac{|I_l| - |I_r|}{2} \right| \leq |t_k - t_0| + \frac{|I_l| + |I_r|}{2} \sim 2|I| \]
    again by Theorem \ref{teo.aprioribounds}.

    To see that any angle between the top side of $Q$ and either the right or left sides of $Q$ is bounded away from $0$ and $\pi$, it is enough to notice that $y_n(t_0) \asymp y_n(t_k) \asymp |I|$. Indeed, the polygon formed by the points $t_0, t_k, z_n(t_k), z_n(t_0)$ has all comparable sides (by the above arguments) and two right angles (at the base), and therefore cannot be ''arbitrarily far from being a square'' --- more formally, if the angle was not bounded, the top side would not be comparable to the left and right sides.

    Finally, to see that the angles between adjacent bottom sides of $Q$ are bounded away from $0$ and $2\pi$, we simply observe that, given $t_{j-1}, t_j, t_{j+1} \in I\cap\mathcal{Q}_{n+1}$ adjacent, the polygon obtained by taking as vertices $t_{j-1}, t_{j+1}, z_{n+1}(t_{j+1}), z_{n+1}(t_j), z_{n+1}(t_{j-1})$ has all comparable sides, with two vertical ones making right angles with the bottom, and the previous arguments conclude the proof.
\end{proof}

\begin{rmk}
    
Up to this point, we have overlooked a certain edge-case: when $y_n(t) = y_{n+1}(t)$ for some $t \in \mathcal{Q}_n$; this can only happen when $a_{n+1} = 1$. This would imply that a cell with $z_n(t)$ as a vertex has to be a triangle. We have not dealt with this case carefully, but the result still holds by completely analogous proofs. This implies that triangular cells cannot distort too much, meaning that any affine map taking a triangle of $\Gamma$ to a corresponding one of $\Gamma'$ has distortion bounded by a universal constant --- not depending on the actual cell we take. Indeed, by triangulating cells, we see that the same holds if we consider cells with up to a certain number of sides. This is useful for us, since some methods we will employ will usually assume that the cell has a large enough number of sides, or is in a deep enough level of the grid.

\end{rmk}

\subsection{Uniformization mappings, the dynamics of Möbius maps and the Yoccoz extension}

Let us now fix a cell $Q$ of $\Gamma$ and its image $Q'$ in $\Gamma'$; let us also fix their base intervals $I = [t_0, t_k]$ and $I' = [t'_0, t'_k]$. We are looking to extend the piece-wise affine map between $\partial Q$ and $\partial Q'$ to a map between $\overline{Q}$ and $\overline{Q'}$. The geometric control lemma we just proved allows us to simplify the problem by turning these polygons into the upper half-plane with good control on the distortion. Indeed, let $S := \{ x + iy \ | \ |x| \leq 1, 0 \leq y \leq 2 \}$ be a standard square, and define $\phi: Q \to S, \phi': Q' \to S$ as maps that are affine on the real coordinate and affine on each vertical fiber. To make things more explicit, we can consider the top sides of $Q$ and $Q'$ as graphs of affine functions $g_1$ and $g'_1$, while the bottom sides are graphs of piece-wise affine functions $g_2$ and $g'_2$. Then
\[ \phi(x + iy) = \left( -\frac{t_k - x}{t_k - t_0} + \frac{x - t_0}{t_k - t_0} \right) + 2i\left( \frac{y - g_2(x)}{g_1(x) - g_2(x)} \right) \]
and $\phi'$ is given by the same expression, replacing $t_0, t_k, g_1, g_2$ by $t'_0, t'_k, g'_1, g'_2$. By Lemma \ref{lema.geometric_control}, we get that $g_1, g_2, g'_1$ and $g'_2$ all have uniformly bounded derivatives, since their slopes are ratios between adjacent intervals of the partition. Therefore $\phi$ and $\phi'$ have uniformly bounded distortions --- i.e. their distortions do not depend on the actual cell $Q$.

\begin{figure}
    \centering
    \includegraphics[width=0.9\linewidth]{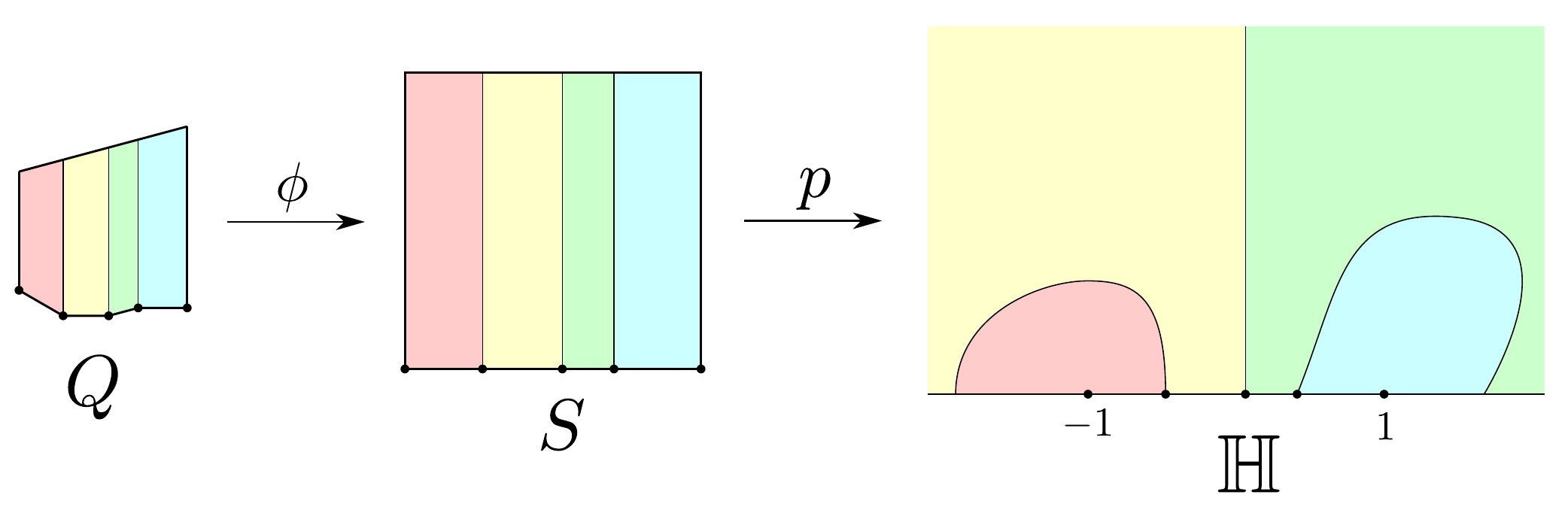}
    \caption{Uniformization of a cell $Q$ into the upper half-plane. Notice that the distribution of points of the partition on the base remains the same.}
    \label{fig:cell_uniformization}
\end{figure}

To get to the upper half-plane now, we will fix a quasiconformal map $p: S \to \h$ that coincides with identity on the interval $[-1, 1]$; this can be done by first taking a conformal uniformization $p_1: S \to \h$, then a quasisymmetric map on $\R$ that coincides with $\phi^{-1}$ on $[-1, 1]$, then a quasiconformal extension $p_2$ of that quasisymmetric map, and finally setting $p := p_2\circ p_1$. As long as we use the same map $p$ for all cells, the final result of this construction has turned the previous problem of extending the piece-wise affine map from $\partial Q$ to $\partial Q'$ into extending the piece-wise affine homeomorphism of the real line that is identity outside of $[-1, 1]$ and sends the points $s_j := \phi(z_{n+1}(t_j))$ to the points $s'_j := \phi'(z'_{n+1}(t'_j))$. To solve this problem, we will need to understand the distribution of these points. Since the points $t_j'$ are obtained by rotation, their distribution is uniform and then we get the following control:

\begin{lema}\label{lema.uniform_bound_rotation}

    Let $I' = [t'_0, t'_k]$ be an interval of the partition $\mathcal{Q}'_n$ and write $I'\cap \mathcal{Q}'_{n+1} = \{ t'_0 < t'_1 < \dots < t'_k \}$. Then
    \[ t'_j - t'_{j-1} \asymp \frac{t'_k - t'_0}{k}, \ j = 1, \dots, k. \]
    In particular, if $s'_j := \phi'(z'_{n+1}(t'_j))$, we have
    \[ s'_j - s'_{j-1} \asymp \frac{1}{k}, \ j = 1,\dots, k. \]

\end{lema}

The distribution of the points $t_j$ is a much more involved problem. When we subpartition the interval $I$, we get intervals that are the images of their adjacent under the map $f^{q_n}$. Since $f$ is analytic, it has negative Schwarzian derivative, which implies that $f^{q_n}|_{I}$ is an \textit{almost-parabolic map}. The following estimates were obtained by Yoccoz for this class of maps (see Appendix B in \cite{deFaria-deMelo1999} or Lemma 7.3 in \cite{deFaria-Guarino} for a precise definition of almost-parabolicity and the proof of the estimates):

\begin{lema}[Yoccoz's Almost-Parabolic Bound]\label{teo.yoccoz_almost-parabolic_bound}
    
    Let $I = [t_0, t_k]$ be an interval of the partition $\mathcal{Q}_n$ and write $I\cap \mathcal{Q}_{n+1} = \{ t_0 < t_1 < \dots < t_k \}$. Then
    \[ t_j - t_{j-1} \asymp \frac{t_k - t_0}{\min\{j, k + 1 - j\}^2}, \ j = 1, \dots, k. \]
    In particular, if $s_j := \phi(z_{n+1}(t_j))$, we have
    \[ s_j - s_{j-1} \asymp \frac{1}{\min\{j, k + 1 - j\}^2}, \ j = 1, \dots, k. \]
    
\end{lema}

The main trick now is to notice that the map we want to extend sends a distribution concentrated in the center of the interval to one that is uniform; or, equivalently, the inverse of that map sends a uniform distribution to one that accumulates in the center. Dynamically, this picture is similar to one where we have an attracting fixed point in the interior of the interval and repelling fixed points at both ends. We will use this observation to change the problem of extending the piece-wise affine map, which has $k$ affine pieces, to that of extending a piece-wise M\"{o}bius map, which will have only two pieces to be considered. The M\"{o}bius maps in particular will be modeled by the family
\[ \zeta_a(z) := \frac{z}{a - (a-1)z}, \ a \geq 2. \]
We can see that $\zeta_a$ fixes $0$ and $1$, with multipliers $1/a$ and $a$ at each of them, respectively. Furthermore, $\zeta_a$ leaves the interval $[0, 1]$ invariant.

\begin{prop}\label{prop.mobius_family_properties}

    The derivative of $\zeta_a$ on $[0, 1]$ is increasing. For $x, x+\varepsilon \in [0, 1], \varepsilon > 0$, one has
    \begin{equation}
        1 < \frac{\zeta_a'(x + \varepsilon)}{\zeta_a'(x)} \leq (1 +\varepsilon a)^2,
    \end{equation}
    \begin{equation}
        \frac{\varepsilon^3a}{(1 + \varepsilon a)^2}\frac{1}{(1 - x)^2} \leq \zeta_a(x + \varepsilon) - \zeta_a(x) \leq \frac{\varepsilon(1 + \varepsilon a)^2}{a}\frac{1}{(1 - x)^2}.
    \end{equation}

\end{prop}
\begin{proof}

    It is not hard to verify that $\zeta_a'(x) = a/(a - (a - 1)x)^2$ is increasing, giving us the lower bound in (2.1). The upper bound of (2.1) is equivalent to
    \[ \frac{a - (a - 1)x}{a - (a - 1)(x + \varepsilon)} \leq 1 + \varepsilon a \iff (a - 1)\varepsilon + a(a - 1)\varepsilon(x + \varepsilon) \leq \varepsilon a^2 \]
    which is easily seen to be true since $x + \varepsilon \leq 1$. For the second set of estimates, we apply the mean value theorem to get
    \[ \zeta_a'(x) \leq \frac{\zeta_a(x + \varepsilon) - \zeta_a(x)}{\varepsilon} \leq \zeta_a'(x + \varepsilon). \]
    Now it is easy to see that
    \begin{align*}
        \zeta_a'(x + \varepsilon) = \frac{a}{(a - (a - 1)(x + \varepsilon))^2} & \leq \frac{(1 + \varepsilon a)^2}{a(1 - x)^2} \iff \\
        \frac{a}{a - (a - 1)(x + \varepsilon)} & \leq \frac{1 + \varepsilon a}{1 - x} \iff \\
        0 & \leq - \varepsilon a + x + \varepsilon + \varepsilon a(a(1 - x) - \varepsilon a + x + \varepsilon) \\
         & = x + \varepsilon + \varepsilon a(a - 1)(1 - (x + \varepsilon)),
    \end{align*}
    and
    \begin{align*}
        \zeta_a'(x) = \frac{a}{(a - (a - 1)x)^2} & \geq \frac{\varepsilon^2 a}{(1 + \varepsilon a)^2(1 - x)^2} \iff \\
        \frac{1}{a - (a - 1)x} & \geq \frac{\varepsilon}{(1 + \varepsilon a)(1 - x)} \iff \\
        1 - (1 + \varepsilon)x & \geq 0.
    \end{align*}
    The first inequality is then seen to be true, while the second one will follow from the observation that, since $x, x + \varepsilon \leq 1$, we get $(1 + \varepsilon)x \leq (2 - x)x \leq 1$.
    
\end{proof}

\begin{rmk}

    From these estimates, we can see that, whenever $\varepsilon \asymp 1/a$, we get
    \[ \zeta_a'(x + \varepsilon) \asymp \zeta_a'(x), \]
    \[  \zeta_a(x + \varepsilon) - \zeta_a(x) \asymp \frac{1}{a^2(1 - x)^2}, \]
    in the sense that the constants involved in both estimates depend only on the ratio between $\varepsilon$ and $1/a$. More explicitly, if $C \geq 1$ satisfies $1/C \leq a\varepsilon \leq C$, then
    \[ 1 < \frac{\zeta_a'(x + \varepsilon)}{\zeta_a'(z)} \leq (1 + C)^2, \]
    \[ \frac{1}{C^5} \leq \frac{\zeta_a(x + \varepsilon) - \zeta_a(x)}{1/a^2(1 - x)^2} \leq C(1 + C^2). \]
    
\end{rmk}

\vspace{0.9cm}

Now, let us decompose the integer $k = a + b$, with $a = \lfloor k/2 \rfloor$ and $b = \lceil k/2 \rceil$, and consider the piecewise M\"{o}bius map $\psi: \R \to \R$ defined by
\[ \psi(x) := \begin{cases} s_a' - (1 + s_a')\zeta_a\left( \frac{-x + s_a'}{1 + s_a'} \right) & \text{if } -1 \leq x \leq s_a'; \\ s_a' + (1 - s_a')\zeta_b\left( \frac{x - s_a'}{1 - s_a'} \right) & \text{if } s_a' \leq x \leq 1; \\ x & \text{otherwise.} \end{cases} \]
Notice that on $[-1, s_a']$ the map $\psi$ is conjugate with $\zeta_a$, and on $[s_a', 1]$ with   $\zeta_b$. This is done in such a way that the fixed points of $\zeta_a$ and $\zeta_b$ are taken to the endpoints $1, s_a', -1$ of the intervals. More specifically, $-1$ and $1$ are the repelling fixed points of multipliers $a$ and $b$, respectively, and $s_a'$ is the attracting fixed point of multipliers $1/a$ to the left and $1/b$ to the right. The map $\psi$ admits an explicit extension, which allows us to easily compute its distortion.

\begin{lema}\label{lema.strebel}

    The map $\psi: \R \to \R$, as defined above for a cell $Q$ of level $n$ of the grid $\Gamma$, can be extended to a $K_n$-quasiconformal map of $\h$, with
    \[ K_n = \frac{4}{\pi^2}\log^2a_{n+1} + \mathcal{O}(\log a_{n+1}) \ \text{ as } a_{n+1} \xrightarrow{} \infty. \]

\end{lema}
\begin{proof}

    First, notice that the action of $\psi$ on the interval $[-1, s'_a]$ is the restriction of a M\"{o}bius transformation with real coefficients; let us denote this transformation by $\tilde{\zeta}_a$. This means that $\tilde{\zeta}_a$ is an automorphism of the upper half-plane, and as such must preserve geodesics. If we let $\gamma$ be the geodesic on $\h$ connecting $-1$ to $s'_a$, then, because $\tilde{\zeta}_a$ fixes its endpoints, we must have that it leaves $\gamma$ invariant as well. Notice that $\gamma$ touches the real line at angles of $\pi/2$. Now, consider $L$ a M\"{o}bius transformation mapping $-1$ to $0$ and $s'_a$ to $\infty$, so that $L$ conjugates $\tilde{\zeta}_a$ with another M\"{o}bius action that fixes $0$ and $\infty$, i.e. just multiplication by some scalar. Since $L$ must preserve the multiplier of $-1$, we have that $L\circ \tilde{\zeta}_a\circ L^{-1}(z) = az$. We can choose $L$ such that the interval $[-1, s'_a]$ is sent to the non-negative real line, so that $L(\gamma)$ will be an arc of a circle (in the sphere), connecting $0$ and $\infty$, invariant under multiplication by $a$, and crossing the real line in an angle of $\pi/2$. Since $L$ preserves orientation, $L(\gamma)$ is the non-negative imaginary line. Further conjugating with the exponential map shows us that the action of $\tilde{\zeta}_a$ on the half-disk $D_a$ bounded by $\gamma$ and the interval $[-1, s'_a]$ is conformally conjugate to that of translation by $\log a$ on the strip $\{ x + iy \ | \ 0 \leq y \leq \pi/2 \}$.

    \begin{figure}
        \centering
        \includegraphics[scale = 0.4]{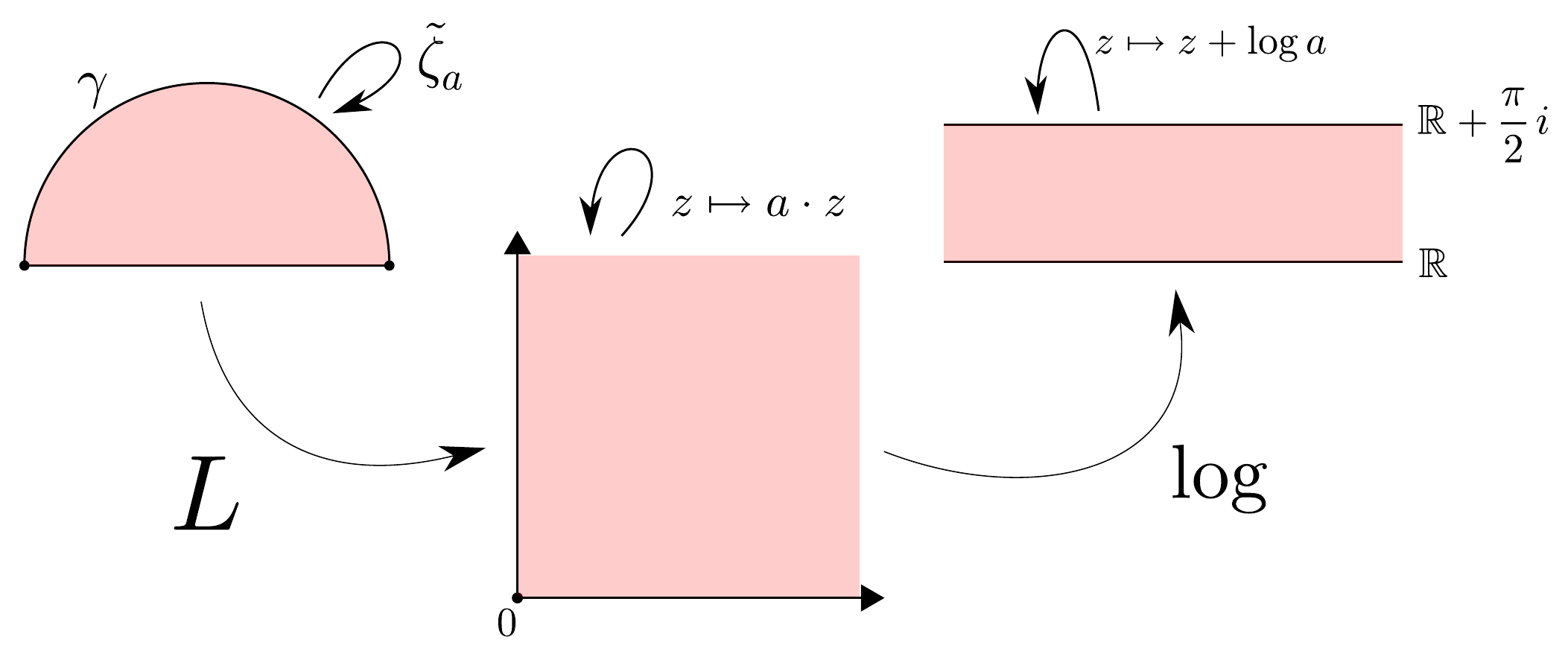}
        \caption{The uniformization of the half-disk between the interval $[-1, s'_a]$ and the geodesic $\gamma$ into an infinite strip of height $\pi/2$. The map $\tilde{\zeta}_a$ is then conjugate to translation by $\log a$.}
        \label{fig:mobius-identity-interpolation}
    \end{figure}

    The problem thus reduces to finding a map on this strip that coincides with translation by $\log a$ on the real line, and the identity on the line $\R + \pi i/2$ --- since we can then use the above uniformization from the half-disk to the strip to conjugate to a map that coincides with $\tilde{\zeta}_a$ on $[-1, s'_a]$ and identity on $\gamma$. This map can be given explicitly:
    \[ \Psi_a(x + iy) = x + iy + \frac{\pi}{2}y\log a. \]
    This is an $\R-linear$ map with constant distortion
    \[ K_{\Psi_a} = \frac{ \sqrt{1 + \frac{1}{\pi^2}\log^2 a} + \frac{1}{\pi}\log a }{ \sqrt{1 + \frac{1}{\pi^2}\log^2 a} - \frac{1}{\pi}\log a } = \frac{4}{\pi^2}\log^2 a + \mathcal{O}(1). \]
    The exact same arguments can be repeated for the interval $[s'_a, 1]$ to obtain an extension $\Psi_b$ with distortion $4\log^2 b/\pi^2 + \mathcal{O}(1)$. Thus, the extension of $\psi$ by taking the conformal conjugates of $\Psi_a$ and $\Psi_b$ on the half-disks $D_a$ and $D_b$ --- the second being that bounded by $[s'_a, 1]$ and the geodesic on $\h$ connecting $s'_a$ to $1$ --- and identity outside will have distortion bounded by some constant (since $b \geq a$)
    \[ K_n = \frac{4}{\pi^2}\log^2 b + \mathcal{O}(1) = \frac{4}{\pi^2}\log^2 a_{n+1} + \mathcal{O}(\log a_{n+1}). \]

\end{proof}

The following corollary to the above lemma will be useful us.

\begin{cor}\label{cor.optimal_distortion_interpolation}

    There exists a universal constant $C \geq 1$, nor depending on the cell $Q$ nor the level $n$, for which $\psi$ is not $((\log^2 a_{n+1}/C) - \varepsilon)$-quasiconformal, for all $\varepsilon > 0$.
    
\end{cor}
\begin{proof}

    Observe that the distortion $K_\psi$ of $\psi$ is constant in the region $D_b$, and thus
    \[ \esssup_{z \in \h} K_\psi(z) = K_{\Psi_b} \asymp \log^2 a_{n+1}. \]
    This means that there exists some universal constant $C \geq 1$ such that
    \[ K_\psi(z) \geq \frac{\log^2 a_{n+1}}{C} \ \forall z \in D_b, \]
    meaning the map $\psi$ cannot be $((\log^2 a_{n+1}/C) - \varepsilon)$-quasiconformal.
    
\end{proof}

Recall that the points $s_j$ and $s'_j$ are the images of the bottom vertices of the cells $Q$ and $Q'$ after uniformization, and we have thus far reduced the problem of extending the piece-wise affine map between $\partial Q$ and $\partial Q'$ to the interiors of the cells, to that of extending the piece-wise affine homeomorphism of $\R$ that sends the points $s_j$ to the points $s'_j$. To that end, consider the piece-wise affine maps $\eta_1:\R \to \R$, defined by sending each $s_j$ to $\psi(s'_j)$, and $\eta_2: \R \to \R$, defined by sending each $\psi(s'_j)$ to $s'_j$, and set $\eta := \psi\circ \eta_2\circ \eta_1$. By showing that this map $\eta$ admits a quasiconformal extension of uniformly bounded distortion, we will be able to transfer the distortion control of the extension of $\psi$ described above back to our problem.

\begin{lema}\label{lema.eta_distortion}

    The map $\eta$ satisfies $\eta'(x) \asymp 1$ wherever the derivative is well defined.

\end{lema}
\begin{proof}

    We will actually show that this is true for $\eta_1$ and $\psi\circ \eta_2$ independently. For $\eta_1$, we remember that, from Lemma \ref{teo.yoccoz_almost-parabolic_bound}, one has
    \[ s_j - s_{j-1} \asymp \frac{1}{\min\{ j, k + 1 - j \}^2}. \]
    It is enough to show that the points $\psi(s'_j)$ satisfy the same property. Indeed, let us first take $j \leq a$, so that $s'_j, s'_{j-1} \leq s'_a$. Then
    \[ \psi(s'_j) - \psi(s'_{j-1}) = (1 + s'_a)\left[ \zeta_a\left( \frac{s'_a - s'_{j-1}}{1 + s'_a} \right) - \zeta_a\left( \frac{s'_a - s'_j}{1 + s'_a} \right) \right]. \]
    Since we have, from Lemma \ref{lema.uniform_bound_rotation}, 
    \[ \frac{s'_a - s'_{j-1}}{1 + s'_a} = \frac{s'_a - s'_j}{1 + s'_a} + \frac{s'_j - s'_{j-1}}{1 + s'_a}, \]
    and
    \[ \frac{s'_j - s'_{j-1}}{1 + s'_a} \asymp \frac{1}{k} \asymp \frac{1}{a}, \]
    we conclude using Proposition \ref{prop.mobius_family_properties} (or rather the remark right after it) that
    \[ \psi(s'_j) - \psi(s'_{j-1}) \asymp \frac{1}{a^2}\left( 1 - \frac{s'_a - s'_j}{1 + s'_a} \right)^{-2} \asymp \frac{1}{j^2} \]
    since $1 + s'_j = s'_j - s'_0 \asymp j/k$. The case when $j > a$ is analogous and uses the observation that $1 - s'_j \asymp (k - j)/k$.

    For the composition $\psi\circ \eta_2$, we get by the chain rule that
    \[ D(\psi\circ \eta_2)(z) = D\psi(\eta_2(z))\cdot D\eta_2(z) \]
    and by the definition of $\eta_2$, we have
    \[ D\eta_2(z) = \frac{s'_j - s'_{j-1}}{\psi(s'_j) - \psi(s'_{j-1})} \asymp \frac{\min\{j, k + 1 - j\}^2}{k} \]
    whenever $z \in [\psi(s'_{j-1}), \psi(s'_j)]$. Let us now assume that $j \leq a$ --- the case $j > a$ will again be completely analogous. Since $\psi$ is conjugated to $\zeta_a$ by an affine map in the interval $[-1, s'_a]$, we get that
    \[ D\psi(\eta_2(z)) = D\zeta_a\left( \frac{-\eta_2(z) + s'_a}{1 + s'_a} \right) \]
    whenever $z \in [\psi(s'_{j-1}), \psi(s'_j)]$. By Proposition \ref{prop.mobius_family_properties}, we conclude that
    \[ D\zeta_a\left( \frac{-s'_j + s'_a}{1 + s'_a} \right) \leq D\psi(\eta_2(z)) \leq D\zeta_a\left( \frac{-s'_{j-1} + s'_a}{1 + s'_a} \right) \asymp D\zeta_a\left( \frac{-s'_j + s'_a}{1 + s'_a} \right). \]
    By direct computations, we see that
    \[ D\zeta_a\left( \frac{-s'_j + s'_a}{1 + s'_a} \right) = \frac{a}{\left( a - (a - 1)\frac{-s'_j + s'_a}{1 + s'_a} \right)^2} \asymp \frac{k}{\left( k - (k - 1)\frac{k - j}{k} \right)^2} \asymp \frac{k}{j^2} \]
    concluding that $D(\psi\circ \eta_2)(z) \asymp 1$.

\end{proof}
\begin{cor}

    The extension $\eta(x + iy) := \eta(x) + iy$ is $K$-quasiconformal, with $K \asymp 1$.
    
\end{cor}

We are finally ready to define the Yoccoz extension.

\begin{deft}

    The \textit{Yoccoz extension} $\Y(h)$ of $h$ is the homeomorphic extension to the piece-wise affine map from $\Gamma$ to $\Gamma'$ determined in each cell $Q$ of the grid $\Gamma$ as the composition
    \[ \Y(h)|_Q = \phi'^{-1}\circ p^{-1}\circ \psi^{-1}\circ \eta\circ p\circ \phi. \]
    
\end{deft}

\begin{figure}
    \centering
    \includegraphics[width=0.8\linewidth]{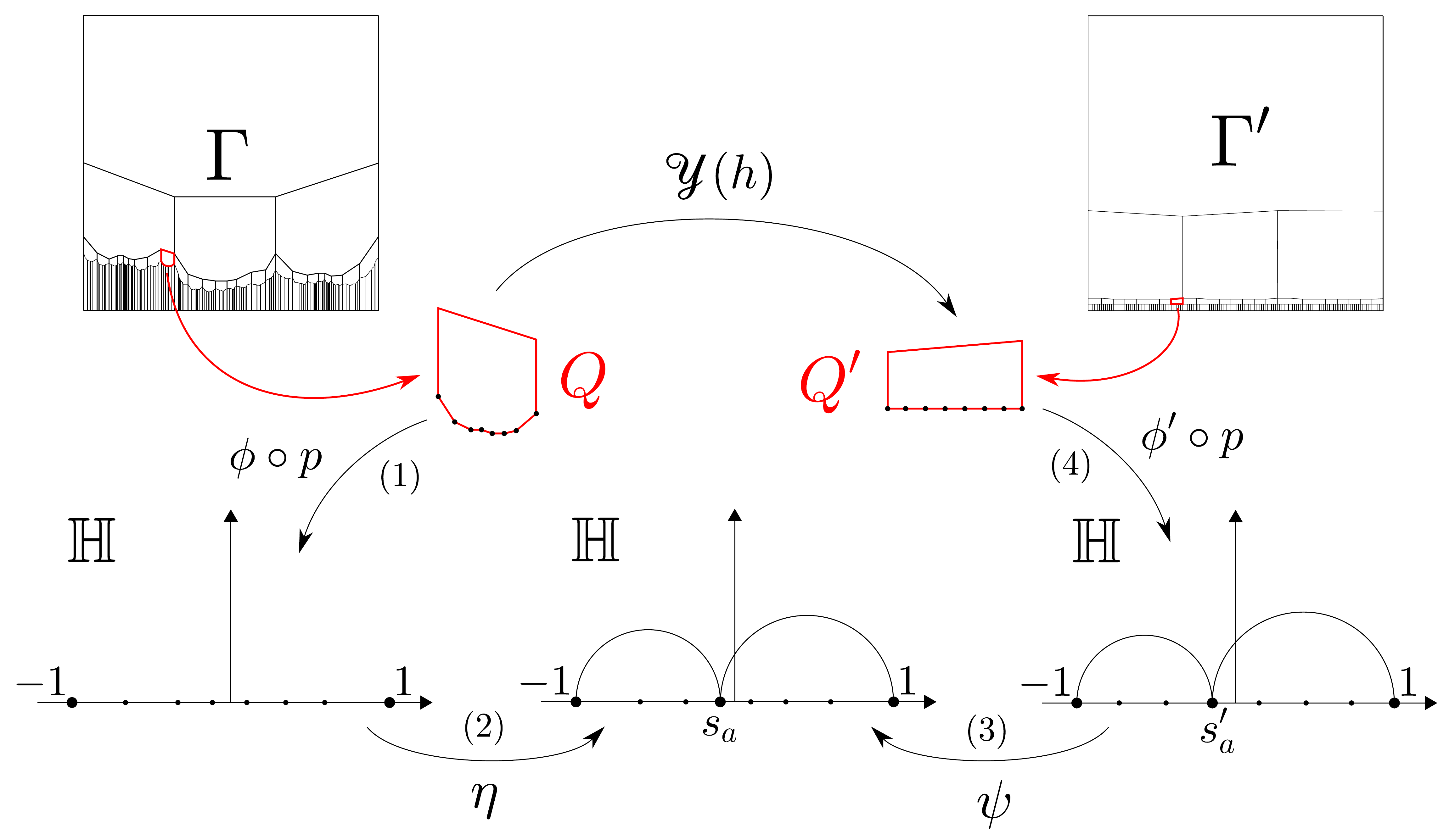}
    \caption{All of the steps to define the Yoccoz extension on a cell $Q$. Step (1) uniformizes the cell into the upper half-plane, maintaining the distribution of base points; step (2) redistributes these points in a controlled way (bounded derivative) so that they are now the image of a uniform distribution under the piecewise M\"{o}bius map; step (3) applies the (inverse of the) piecewise M\"{o}bius map, which has repelling points at the extremes $\pm 1$, and an attracting point at $s_a'$; step (4) applies the (inverse of the) uniformization for the cell $Q' = \Y(h)(Q)$.}
    \label{fig:extension_full_picture}
\end{figure}

\begin{teo}\label{teo.yoccoz_distortion}

    Let $Q$ be a cell of level $n$. Then the map $\Y(h)$ satisfies
    \[ \esssup_{z \in Q} K_{\Y(h)}(z) \asymp \log^2 a_{n+1}. \]
    
\end{teo}

The proof really comes from all of the controls we obtained for the maps involved in the definition, and the following simple observation about quasiconformal maps:

\begin{lema}\label{lema.distortion_composition_lower_bound}

    Let $g_1, g_2, g_3$ be $K_1, K_2, K_3$-quasiconformal, respectively. Assume that $g_2$ is not $(K_2 - \varepsilon)$-quasiconformal for any $\varepsilon > 0$. Then the composition $g_1\circ g_2\circ g_3$ is not $\left( \frac{K_2}{K_1K_3} - \varepsilon \right)$-quasiconformal for any $\varepsilon > 0$.
    
\end{lema}
\begin{proof}

    Since $g_2$ is not $(K_2 - \varepsilon)$-quasiconformal, we must have that there is some quadrilateral $A$ such that $m(g_2(A)) > (K_2 - \varepsilon)m(A)$, where $m$ denotes the modulus. Thus, we find that
    \[ m(g_1\circ g_2\circ g_3(g_3^{-1}(A))) \geq \frac{1}{K_1}m(g_2\circ g_3(g_3^{-1}(A))) > \frac{K_2 - \varepsilon}{K_1}m(g_3(g_3^{-1})(A)) > \frac{K_2 - \varepsilon}{K_1K_3}m(g_3^{-1}(A)) \]
    which implies the claim.
    
\end{proof}

\begin{proof}[Proof of Theorem \ref{teo.yoccoz_distortion}]

    Obtaining $\log^2 a_{n+1}$ as an upper bound comes from the construction itself, since all of the maps involved have universally bounded distortion, except $\psi^{-1}$, whose distortion is bounded by (something comparable to) $\log^2 a_{n+1}$ --- and remember that when two maps are $K_1$ and $K_2$ quasiconformal, respectively, their composition is $K_1K_2$ quasiconformal. By Corollary \ref{cor.optimal_distortion_interpolation}, the map $\psi^{-1}$ is not $((\log^2 a_{n+1}/C) - \varepsilon)$-quasiconformal for some universal constant $C \geq 1$. Lemma \ref{lema.distortion_composition_lower_bound} thus tells us that there is another $C' \geq 1$ such that $\Y(h)$ is not $((\log^2 a_{n+1}/C') - \varepsilon)$-quasiconformal ($C'$ is the product of $C$ with the distortions of the maps $\phi'^{-1}\circ p$ and $\eta\circ p\circ \phi$).
    
\end{proof}

\section{Proof of the main results}

In this section, we prove Theorem \ref{teo.main}.

\subsection{Supporting lemmas}

Before proving the main results, we need to establish some supporting lemmas: the first three give lower bounds on area controls, which are relevant when using the regularity of the extension to obtain information about the rotation number; the fourth is an analogue of Corollary \ref{cor.aprioribounds} for the rotation; and the fifth is a finer inequality relating the values of $q_n$ with the terms $a_n$.

\begin{lema}\label{lema.cell_area_lower_bound}

    Let $Q$ be a cell of the grid $\Gamma$ at level $n$. Then $\Area(Q) \gtrsim \varepsilon^{2n}$ for $\varepsilon \in (0, 1)$ the constant from Corollary \ref{cor.aprioribounds}.

\end{lema}
\begin{proof}

    Let $I$ be the interval of the partition at level $n$ that bounds $Q$. A consequence of Lemma \ref{lema.geometric_control} is that $Q$ must contain a square of side $\ell \asymp |I|$. Since $|I| \gtrsim \varepsilon^n$ by Corollary \ref{cor.aprioribounds}, and the area of such a square is $\ell^2$, we conclude the result.
    
\end{proof}

\begin{lema}\label{lema.area_proportion_bound}

    Let $Q$ be a cell of the grid $\Gamma$ at level $n$. Then there are universal constants $C \geq 1$ and $\lambda \in (0, 1)$ not depending on $Q$ nor $n$ such that
    \[ \Area\left\{ z \in Q \ \left| \ K_{\Y(h)}(z) \geq \frac{\log^2a_{n+1}}{C} \right.\right\} \geq \lambda\Area(Q). \]

\end{lema}
\begin{proof}

    Let us recall that the Yoccoz extension $\Y(h)$ on the cell $Q$ is obtained as a composition of maps
    \[ \Y(h) = \phi'^{-1}\circ p^{-1}\circ \psi^{-1}\circ \eta\circ p\circ \phi, \]
    where $\phi, \phi', p, \eta$ are all quasiconformal with universally bounded distortion, and the map $\psi$ is $K_n$-quasiconformal, with $K_n = 4\log^2a_{n+1}/\pi^2 + \mathcal{O}(\log a_{n+1})$, by Lemma \ref{lema.strebel}. The distortion of $\psi$ is in fact constant in the regions $D_a, D_b$ defined in the proof of that Lemma, and by combining Corollary \ref{cor.optimal_distortion_interpolation} and Lemma \ref{lema.distortion_composition_lower_bound}, we conclude that $\Y(h)|_Q$ has distortion bounded from below by $\log^2 a_{n+1}/C$ for some universal constant $C$ in the region $(\eta\circ p\circ \phi)^{-1}(D_a\cup D_b)$. All we must do now is show that this region occupies a constant proportion of the area of $Q$.

    Since $D_a$ and $D_b$ are halves of disks with centers in $[-1, 1]$ and diameters adding up to $2$, we must have that $D_a\cup D_b$ always contains a half disk of radius $1/2$ --- either the one centered at $-1/2$ or the one centered at $1/2$, and thus $\Area(D_a\cup D_b) \geq \pi/4$. Now, since $\eta'(x) \asymp 1$ along the real line, we must have that $\eta$ distorts areas in $\h$ in a controlled way, i.e. there is a constant $A_1$ not depending on $Q$ nor $n$ such that $\Area(\eta^{-1}(D_a\cup D_b)) \geq A_1$. Since $p$ is fixed and equal to identity in the interval $[-1, 1]$, we must have another constant $A_2$ not depending on $Q$ nor $n$ such that $\Area(p^{-1}\circ \eta^{-1}(D_a\cup D_b)) \geq A_2$. Let now $\lambda = A_2/\Area(S) = A_2/4$, so that
    \[ \Area(p^{-1}\circ \eta^{-1}(D_a\cup D_b)) \geq \lambda\Area(S). \]
    Since $\phi$ preserves the area proportions of sets (it is affine in vertical fibers, i.e. it preserves length proportions in vertical fibers, and we can apply Fubini's theorem), we conclude that
    \[ \Area(\phi^{-1}\circ p^{-1}\circ \eta^{-1}(D_a\cup D_b)) \geq \lambda\Area(Q). \]

\end{proof}

\textbf{Remark.} Through the exact same arguments, we prove a similar result for the inverse $\Y(h)^{-1}$.

\begin{lema}\label{lema.area_proportion_bound_rotation}

    Let $Q'$ be a cell of the grid $\Gamma'$ at level $n$. Then there are universal constants $C \geq 1$ and $\lambda \in (0, 1)$ not depending on $Q'$ nor $n$ such that
    \[ \Area\left\{ z \in Q' \ \left| \ K_{\Y(h)^{-1}}(z) \geq \frac{\log^2a_{n+1}}{C} \right.\right\} \geq \lambda\Area(Q'). \]
    
\end{lema}

\begin{lema}\label{lema.rotation_aprioribounds}

    Let $\theta \in \R\setminus \Q$, $R_\theta$ be the rotation by $\theta$ on $\R/\Z$, and $I'_n := [1, R_{\theta}^{q_n}(1)]$ the interval of $n^{\text{th}}$ return to $1$. Let also $\mathcal{Q}'$ be the adapted dynamical partition for the rotation $R_\theta$. Then, for any interval $I$ in $\mathcal{Q}'_n$, one has
    \[ \frac{1}{q_{n+1}^2} \lesssim |I'_n| \leq |I| \lesssim \frac{1}{q_n}. \]
    
\end{lema}
\begin{proof}

    The lower estimate comes from the fact that
    \[ |I'_n| = |q_n\theta - p_n| \geq \frac{1}{q_n(q_{n+1} + q_n)} \geq \frac{1}{2q_{n+1}^2}. \]
    For the upper estimate, recall that an interval in $\mathcal{Q}'_n$ is either of the form $[R_\theta^{-j}(1), R_\theta^{-j-q_{n-1}}(1)]$, $0 \leq j < q_n$, in which case it is an $(n-1)^{\text{th}}$ return, or of the form $[R_\theta^{-j}(1), R_\theta^{-j-q_n}(1)]\cup [R_\theta^{-j-q_n}(1), R_\theta^{-j-q_n+q_{n-1}}(1)]$, $0 \leq j < q_{n-1}$, in which case it is the union of an $n^{\text{th}}$ return and an $(n-1)^{\text{th}}$ return. From that, we get
    \[ |I| = |q_{n-1}\theta - p_{n-1}| = q_{n-1}\left| \theta - \frac{p_{n-1}}{q_{n-1}} \right| \leq q_{n-1}\frac{1}{q_nq_{n-1}} = \frac{1}{q_n} \]
    in the first case and
    \[ |I| \leq |q_{n-1}\theta - p_{n-1}| + |q_n\theta - p_n| \leq \frac{2}{q_n} \]
    in the second case.
    
\end{proof}

Finally, the following Lemma is a refinement of the classical inequalities
\[ \prod_{k=1}^n b_k \leq q_n \leq \prod_{k=1}^n (a_k + 1), \]
which come easily by induction from the relation $q_{n+1} = a_{n+1}q_n + q_{n-1}$. The lower inequality is not useful in its current state, since it cannot see long strings of ones in the sequence $a_n$. To fix it, we will slightly addapt the sequence $a_n$.

\begin{lema}\label{lema.qn_an_comparison}

    Let $\theta = [a_1, a_2, \dots]$ and $\{ q_n \}_{n\geq 1}$ be the associated sequence of denominators from the approximants of $\theta$. Define the sequence
    \[ b_k := \begin{cases} a_k & \text{if } a_k \neq 1 \\ \frac{F_\ell}{F_{\ell - 1}} & \text{if } a_{k-i} = 1 \text{ for all } 0 \leq i < \ell \text{ and } a_{k-\ell} \neq 1 \end{cases} \]
    with $F_\ell$ the $\ell^{\text{th}}$ Fibonacci number --- starting from $F_0 = F_1 = 1$.
    Then
    \[ \prod_{k=1}^n b_k \leq q_n \leq \prod_{k=1}^n (a_k + 1). \]
    
\end{lema}
\begin{proof}

    As observed before, we only need to prove the left inequality; we will apply induction on the relation $q_{n+1} = a_{n+1}q_n + q_{n-1}$. First, observe that $q_1 = a_1 = b_1$. If we then know the result for all $m < n$, we consider two possible cases: either $a_n \neq 1$ and $q_n \leq a_nq_{n-1} \leq \prod_{k=1}^n b_k$; or $a_n = a_{n-1} = \dots = a_{n-\ell+1} = 1$ and $a_{n-\ell} \neq 1$, in which case we can see that
    \[ q_n = q_{n-1} + q_{n-2} = 2q_{n-2} + q_{n-3} = \dots = F_\ell q_{n-\ell} + F_{\ell-1} q_{n-\ell-1} \geq F_\ell\prod_{k=1}^{n-\ell}b_k. \]
    By the definition of $b_n$, we see that $F_\ell = \prod_{k=n-\ell+1}^n b_k$ and the result follows.
    
\end{proof}

\subsection{Proof of the main result}

Proving that $\Y(h)$ is a David map under the Petersen-Zakeri conditions involves both a control of its distortion from Theorem \ref{teo.yoccoz_distortion} and a control of the areas of the cells of $\Gamma$, which is obtained as a consequence of Theorem \ref{teo.aprioribounds}. But, in proving the optimality of the condtion for the Yoccoz extension, we need to know that the distortion of the map is not somehow concentrated in small areas of each cell, nor that the areas of the cells decay faster than expected. This is what Lemmas \ref{lema.cell_area_lower_bound} and \ref{lema.area_proportion_bound} are for.

\begin{proof}[Proof of Theorem \ref{teo.main}.a]

    The fact that $\rho(f) \in PZ$ implies that $\Y(h)$ is a David map is already known from \cite{Petersen-Zakeri2004}, but we include a proof for completeness. Let $\hat{K} \geq 1$, and assume $z \in \h$ satisfies $K_{\Y(h)}(z) > \hat{K}$. Then either $z$ is in the grid $\Gamma$, which has zero measure, or $z$ is in some cell of level $n$. Combining Theorem \ref{teo.yoccoz_distortion} with the hypothesis that $\log^2a_n = \mathcal{O}(n)$, we see that there is a universal constant $C \geq 0$ satisfying $K_{\Y(h)}(z) \leq Cn$, and in particular $n \geq \hat{K}/C$. If $n_0 = n_0(\hat{K})$ is the smallest integer such that $n_0 \geq \hat{K}/C$, we find that
    \[ \Area\{ z \in [0, 1]\times [0, 1] \ | \ K_{\Y(h)}(z) \geq \hat{K} \} \leq \Area(\text{cells of level at least }n_0). \]
    From Corollary \ref{cor.aprioribounds}, these cells reach a height of at most $\sigma^{n_0}$ (up to some constant), for some universal $\sigma \in (0, 1)$, and therefore
    \[ \Area\{ z \in [0, 1]\times [0, 1] \ | \ K_{\Y(h)}(z) \geq \hat{K} \} \lesssim \sigma^{n_0} \leq \sigma^{K/C} = e^{-\frac{\log(1/\sigma)}{C}\hat{K}}. \]
    Thus $\Y(h)$ is a David map, with some constant $A \geq 1$ and $\alpha = \log\sigma^{-1}/C$.

    Suppose now $\Y(h)$ is a David map with constants $A, \alpha > 0$. By definition, this means that
    \[ \Area\{ z \in [0,1]\times [0,1] \ | \ K_{\Y(h)}(z) \geq \hat{K} \} \leq Ae^{-\alpha \hat{K}}. \]
    Combining both Lemmas \ref{lema.cell_area_lower_bound} and \ref{lema.area_proportion_bound}, we get
    \[ \varepsilon^{2n} \lesssim \Area(Q) \leq \frac{1}{\lambda}\Area\left\{ z \in Q \ \left| \ K_{\Y(h)}(z) \geq \frac{\log^2a_{n+1}}{C} \right.\right\} \leq \frac{1}{\lambda}Ae^{-\alpha\log^2a_{n+1}/C}. \]
    Taking logarithms on both sides, we find that
    \[ \frac{\alpha}{2C}\log^2a_{n+1} - \frac{1}{2}\log\frac{A}{\lambda} + \mathcal{O}(1) \leq n, \]
    and since $A$ and $\lambda$ do not depend on $n$, we recover that $\rho(f) \in PZ$.

\end{proof}

The ideas involved in item b of Theorem \ref{teo.main} are similar to the ones used up to this point, now needing to apply some geometric control for the rotation instead of the map $f$.

\begin{proof}[Proof of Theorem \ref{teo.main}.b]

    Let us start showing that the condition is necessary. Observe that Lemma \ref{lema.area_proportion_bound_rotation} still guarantees that the area of the points in each cell $Q'$ of the grid $\Gamma'$ where the distortion of $\Y(h)^{-1}$ exceeds some $\log^2 a_{n+1}/C$ is at least a constant proportion of the total area. We also get from Lemma \ref{lema.geometric_control} that $Q'$ will contain some square of side comparable to $|I|$, where $I$ is the interval of the dynamical partition $\mathcal{Q}'_n$ that bounds $Q'$. Thus, if we assume that $\Y(h)^{-1}$ is David with constants $A, \alpha > 0$, then the lower bound from Lemma \ref{lema.rotation_aprioribounds} gets us:
    \[ Ae^{-\alpha\log^2 a_{n+1}/C} \geq \Area\left\{ K_{\Y(h)^{-1}}(z) \geq \frac{\log^2 a_{n+1}}{C} \right\} \geq \lambda\Area Q' \gtrsim \frac{1}{q_{n+1}^4} \geq \prod_{k=1}^{n+1}(a_k + 1)^{-4} \]
    and the desired conclusion follows by taking logarithms (and noticing that $\log(a_{n+1} + 1) = o(\log^2a_{n+1})$).

    Now, we will show that $\Y(h)^{-1}$ is a David map whenever
    \[ \log^2a_{n+1} \lesssim \sum_{k=1}^{n}\log b_k =: B_n. \]
    The equivalence between this condition and the one in the hypothesis is easy to verify. Let $\hat{K} \geq 1$, and assume $z \in \h$ satisfies $K_{\Y(h)^{-1}}(z) > \hat{K}$. Then either $z$ is in the grid $\Gamma'$, which has zero measure, or $z$ is in some cell of level $n$. Combining Theorem \ref{teo.yoccoz_distortion} with the hypothesis that $\log^2a_{n+1} \lesssim B_n$, we see that there is a universal constant $C \geq 0$ satisfying $K_{\Y(h)^{-1}}(z) \leq CB_n$, and in particular $B_n \geq \hat{K}/C$. If $n_0 = n_0(K)$ is the smallest integer such that $B_{n_0} \geq K/C$, we find that
    \[ \Area\{ z \in [0, 1]\times [0, 1] \ | \ K_{\Y(h)^{-1}}(z) \geq \hat{K} \} \leq \Area(\text{cells of level at least }n_0). \]
    From Lemmas \ref{lema.rotation_aprioribounds} and \ref{lema.qn_an_comparison}, these cells reach a height of at most $e^{-B_{n_0}}$ (up to some constant), and therefore
    \[ \Area\{ z \in [0, 1]\times [0, 1] \ | \ K_{\Y(h)^{-1}}(z) \geq \hat{K} \} \lesssim e^{-B_{n_0}} \leq e^{-\hat{K}/C}. \]
    Thus $\Y(h)^{-1}$ is a David map, with some constant $A \geq 1$ and $\alpha = 1/C$.

\end{proof}

By looking at our arithmetic classes, we observe an unexpected phenomenon: the condition for $\Y(h)$ to be David is strictly included in the condition for $\Y(h)^{-1}$ to be David. Indeed, if $\log^2 a_n = \mathcal{O}(n)$, then
\[ \sum_{k=1}^n \log(a_k + 1) \geq n\cdot \log 2 \gtrsim \log^2 a_n. \]
We remind the reader that \textit{biDavid maps} are David maps with David inverse.

Theorem \ref{teo.main}.b implies the following:

\begin{cor}\label{cor}

    If $\Y(h)$ is a David map, then it is actually biDavid.
    
\end{cor}

\textbf{Remark.} Although the condition for $\Y(h)^{-1}$ to be David results in a cumbersome arithmetic class, we can compare it to other classes that are easier to understand. For that, let us denote
\[ PZ_\varepsilon := \left\{ \theta = [a_1, a_2, \dots] \in \R\setminus \Q \ \left| \ \log a_n = \mathcal{O}\left( n^{\frac{1}{2} + \varepsilon} \right) \right.\right\}, \]
and recall that we defined
\[ \mathcal{A} := \left\{ \theta = [a_1, a_2, \dots] \in \R\setminus \Q \ \left| \ \log^2 a_{n+1} \lesssim \sum_{k=1}^n \log(a_k + 1) \right.\right\}. \]
Then we find that
\[ PZ = PZ_0 \subset \mathcal{A} \subset PZ_{1/2}. \]
Indeed, the first inclusion is exactly the previous corollary, and the second comes from a simple induction argument.

Furthermore, we can check that $\mathcal{A}$ intersects $PZ_\varepsilon$ for all $0 < \varepsilon \leq 1/2$ without entirely containing these classes. Indeed, on one hand, taking $\theta = [a_1, a_2, \dots]$ with $a_n$ the integer part of $e^{n^{1/2 + \varepsilon}}$, we have $\theta \in \mathcal{A}\cap PZ_\varepsilon$. On the other hand, if we let $a_n$ be the integer part of $e^{n^{1/2 + \varepsilon}}$ when $n$ is a perfect square, and $1$ otherwise, then $\theta \in PZ_\varepsilon$ still, but $\theta \notin \mathcal{A}$. It is also worth mentioning that $\mathcal{A}$ is still included in the Brjuno class.

\subsection{Generalizations}\label{SG}

The proofs presented in the above section can be adapted to more general classes of maps, as long as they are defined through similar area controls as the David class. To make this more explicit, let us state the following:

\begin{teo}\label{gen}

    Let $F: \s^1 \to \s^1$ be a critical circle map with $\theta = \rho(f) \in \R\setminus \Q$, and let $f$ be a lift of it. Let $h: \R \to \R$ be the lift of the normalized linearizing map, which satisfies $h\circ f = T_{\theta(f)}\circ h$, where $T_{\theta}(z) = z + \theta$. Let $\Y(h): \overline{\h} \to \overline{\h}$ be the Yoccoz extension of $h$. Let $f: [1, \infty) \to [1, \infty)$ be a strictly increasing function satisfying that, for any $C \geq 1$, $f(Cx) \leq C'f(x)$, where $C'$ depends only on $C$. Then:
    \begin{itemize}
        \item [a]. there are constants $A, \alpha > 0$ satisfying
        \[ \Area\left\{ z \in [0, 1]\times [0, 1] \ | \ K_{\Y(h)}(z) \geq f(\hat{K}) \right\} \leq Ae^{-\alpha \hat{K}} \]
        for all $\hat{K} \geq 1$ if and only if
        \[ \log a_n = \mathcal{O}\left( \sqrt{f(n)} \right); \]

        \item [b]. there are constants $A, \alpha > 0$ satisfying
        \[ \Area\left\{ z \in [0, 1]\times [0, 1] \ | \ K_{\Y(h)^{-1}}(z) \geq f(\hat{K}) \right\} \leq Ae^{-\alpha \hat{K}} \]
        for all $\hat{K} \geq 1$ if and only if
        \[ \log^2 a_n \lesssim f\left( \sum_{k=1}^{n}\log(a_k + 1) \right). \]
    \end{itemize}

\end{teo}
\begin{proof}

    Let us assume $\Y(h)$ and $\Y(h)^{-1}$ satisfy the area controls of items a and b. Then, according to Lemmas \ref{lema.cell_area_lower_bound}, \ref{lema.area_proportion_bound}, \ref{lema.area_proportion_bound_rotation}, \ref{lema.rotation_aprioribounds}, \ref{lema.qn_an_comparison}, for values of $n$ with $a_{n+1}$ large enough, there are cells $Q, Q'$ of $\Gamma, \Gamma'$, respectively, and a constant $C \geq 1$, such that
    \begin{align*}
        Ae^{-\alpha f^{-1}\left( \frac{\log^2 a_{n+1}}{C} \right)} & \geq \Area\left\{ z \in Q \ \left| \ K_{\Y(h)}(z) \geq \frac{\log^2a_{n+1}}{C} \right.\right\} \\
         & \geq \lambda\Area(Q) \gtrsim \varepsilon^{2n} \implies \\
        \implies f^{-1}\left( \frac{\log^2 a_{n+1}}{C} \right) & \lesssim n + 1 \implies \log^2 a_{n+1} \lesssim f(n + 1),
    \end{align*}
    and
    \begin{align*}
        Ae^{-\alpha f^{-1}\left( \frac{\log^2 a_{n+1}}{C} \right)} & \geq \Area\left\{ z \in Q' \ \left| \ K_{\Y(h)^{-1}}(z) \geq \frac{\log^2a_{n+1}}{C} \right.\right\} \\
         & \geq \lambda\Area(Q') \gtrsim \prod_{k=1}^{n+1}(a_k + 1)^{-4} \implies \\
        \implies f^{-1}\left( \frac{\log^2 a_{n+1}}{C} \right) & \lesssim \sum_{k=1}^{n+1}\log(a_k + 1) \implies \log^2 a_{n+1} \lesssim f\left( \sum_{k=1}^{n+1}\log(a_k + 1) \right).
    \end{align*}

    On the other hand, if $\theta$ satisfies the condition $\log a_n = \mathcal{O}(\sqrt{f(n)})$, we take $\hat{K} \geq 1$ and observe that, if there is a point $z$ in a cell of level $n$ satisfying $K_{\Y(h)}(z) \geq f(\hat{K})$, then $f(\hat{K}) \leq C\log^2 a_{n+1}$, for some universal $C \geq 1$, from Theorem \ref{teo.yoccoz_distortion}, which in turn implies that $f(\hat{K}) \leq C'f(n)$, for some universal $C' \geq 1$, from the hypothesis. By the condition imposed on $f$, there is $C'' \geq 1$ also universal satisfying $\hat{K} \leq C''n$. Thus, taking $n_0 = n_0(\hat{K})$ the smallest integer such that $n_0 \geq \hat{K}/C''$, we have
    \[ \Area\{ z \in [0, 1]\times [0, 1] \ | \ K_{\Y(h)} \geq f(\hat{K}) \} \leq \Area(\text{cells of level at least }n_0) \lesssim \sigma^{n_0} \leq \sigma^{\hat{K}/C''} \]
    which is the area control we wanted, with some $A \geq 1$ and $\alpha = \log(1/\sigma)/C''$.

    The argument for $\Y(h)^{-1}$ is perfectly analogous.
    
\end{proof}

We can now apply this theorem to two classes of maps.

\begin{deft}
    A homeomorphism $H: \h \to \h$ is \textit{strong David} if there are $A, \alpha > 0$ such that
    \[ \Area\{ z \in [0,1]\times [0,1] \ | \ K_H(z) \geq K \} \leq Ae^{-\alpha e^K}, \]
    which can be rewritten as
    \[ \Area\{ z \in [0,1]\times [0,1] \ | \ K_H(z) \geq \log K \} \leq Ae^{-\alpha K}. \]
\end{deft}

As a consequence of Theorem \ref{gen} we obtain the following

\begin{cor}\label{SD}
    $\Y(h)$ is strong David if and only if $\log a_n = \mathcal{O}(\sqrt{\log n})$, and the inverse $\Y(h)^{-1}$ is strong David if and only if 
    \[ \log^2 a_{n+1} \lesssim \log\sum_{k=1}^{n+1} \log(a_k + 1). \]
\end{cor}
 It is clear that $\Y(h)^{-1}$ is strong David when $\Y(h)$ is, but the converse must also be true, since strong David maps are closed under taking inverses \cite{David1988}. It is also possible to check this fact directly by proving that the arithmetic classes above are actually the same.

\begin{deft}
    A homeomorphism $H: \h \to \h$ is of \textit{finite distortion} if there are $A, \alpha > 0$ and $v$ some concave function satisfying $\int_1^\infty \diff t/tv(t) = \infty$ such that
    \[ \Area\{ z \in [0,1]\times [0,1] \ | \ K_H(z) \geq Kv(K) \} \leq Ae^{-\alpha K}. \]
\end{deft}

As a consequence of Theorem \ref{gen} we obtain the following:

\begin{cor}\label{FD}
    $\Y(h)$ is of finite distortion if and only if $\log a_n = \mathcal{O}(\sqrt{nv(n)})$ for some such $v$, while $\Y(h)^{-1}$ is of finite distortion if and only if
    \[ \log^2 a_{n+1} \lesssim \left( \sum_{k=1}^{n+1} \log(a_k + 1) \right)\cdot v\left( \sum_{k=1}^{n+1} \log(a_k + 1) \right). \]
\end{cor}
This regularity class is used in \cite{Shen2018} to generalize the Petersen-Zakeri result, and in particular the arithmetic class there considered was exactly the optimal one for the Yoccoz extension. \\

\textbf{Remark.} One can find an example where $\Y(h)^{-1}$ is David, but $\Y(h)$ is not even a finite distortion map: define $a_n$ to be the integer part of $e^{\sqrt{n\log n}}$ when $n$ is a perfect square and $1$ otherwise. That gives us that, on one hand,
\[ \sum_{k=1}^n \log(a_k + 1) \lesssim \int_0^{\sqrt{n} + 1} k\sqrt{\log k}\diff k + n\log 2 \asymp n\sqrt{\log n}, \]
and on the other hand $\log^2 a_n \asymp n\log n$.

\subsection{A Remark on Optmality}\label{optimality}

Let us begin by stating the following conjecture made by Petersen and Zakeri in \cite{Petersen-Zakeri2004}
\begin{conjecture}
    If $F:\s^1\to \s^1$ is a critical circle mapping with rotation number $\theta$ and if the normalized linearizing map $H:\s^1\to \s^1$, which satisfies $H \circ F = R_{\theta} \circ H$, admits a David extension $\hat{H} : \overline{\D}\to \overline{\D}$, then $\theta\in PZ$.
\end{conjecture}

If it was possible to prove that the Yoccoz extension is somehow extremal among extensions of the Yoccoz conjugator, one could conclude the conjecture as a direct consequence of Theorem \ref{teo.main}. Since such a notion has no standard definition, however, we wish to compare the Yoccoz extension with other possible extensions of $h$. Specifically, we will consider the maps $h_n$ which are the piece-wise affine approximation of $h$ at level $n$ --- i.e. $h_n(x) = h(x)$ for all $x \in \mathcal{Q}_n$ and $h_n$ is affine between these points; see Figure \ref{fig:yoccoz-approximant}. It turns out that one can estimate their quasisymmetry constants in terms of coefficients $a_n$.

\begin{figure}
    \centering
    \includegraphics[width=0.9\linewidth]{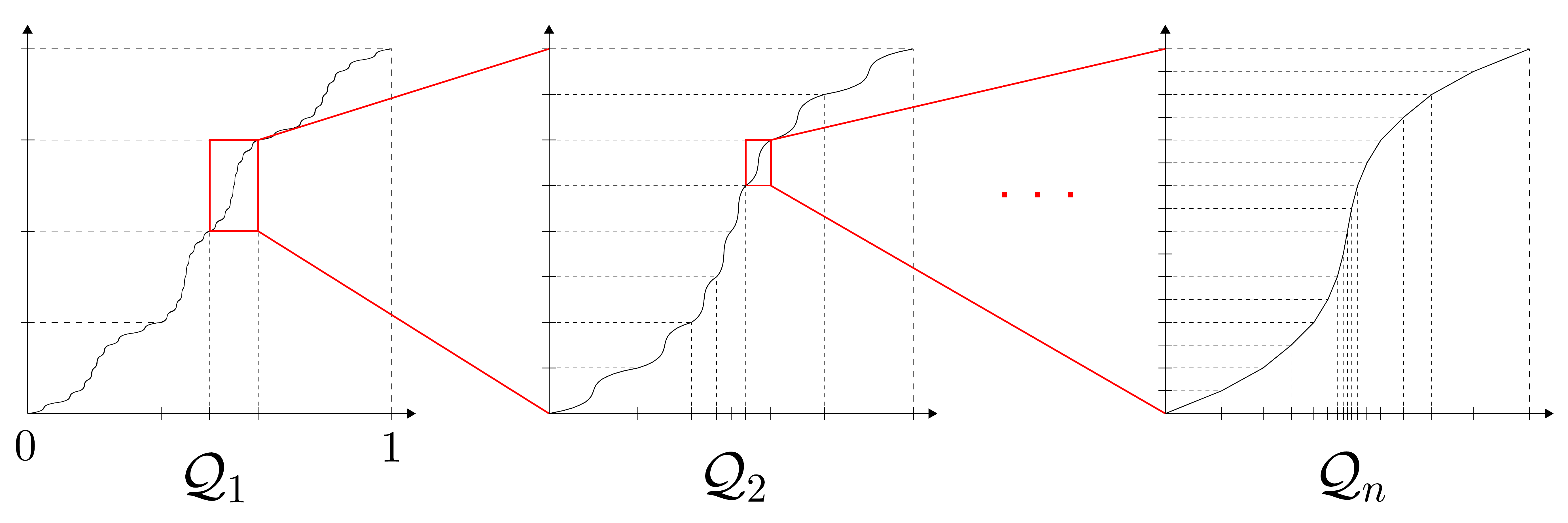}
    \caption{The graph of a map $h_n$ and zoomed in sections emphasizing the distribution of points given by the Yoccoz almost-parabolic inequality.}
    \label{fig:yoccoz-approximant}
\end{figure}

\begin{prop}\label{qs} Let $h_n$ be the piece-wise affine approximation of $h$. Then 
    \[ \|h_n\|_{qs} \gtrsim \max_{1\leq m \leq n} a_m. \]
\end{prop}
\begin{proof}
   Let $a_m$ be the maximum element in the sequence up to term $n$. Take $I$ any interval determined by elements of $\mathcal{Q}_{m-1}$ and we denote by $t_0 < t_1 < \dots < t_k$ the elements of $\mathcal{Q}_m\cap I$ --- so $I = [t_0, t_k]$; recall that $k = a_m$ or $a_{m + 1}$. Let $a := \lfloor k/2 \rfloor$, and set
    \[ t := \frac{t_a + t_0}{2} \ \text{ and } \ \delta := \frac{t_a - t_0}{2}. \]
    Observe that $t - \delta = t_0$ and $t + \delta = t_a$. We first observe that $\#([t_0, t]\cap \mathcal{Q}_m) \asymp 1$. Indeed, if $\#([t_0, t]\cap \mathcal{Q}_m = \nu)$, then
    \[ \frac{t_a - t}{t_k - t_0} = \frac{t_{\nu + 1} - t}{t_k - t_0} + \sum_{j=\nu + 2}^{a} \frac{t_j - t_{j-1}}{t_k - t_0} \lesssim \sum_{j=\nu + 1}^{\infty} \frac{1}{j^2} \xrightarrow[\nu\to\infty]{} 0, \]
    observing that $t_{\nu + 1} - t \leq t_{\nu + 1} - t_\nu$. But from the choice of $t$, $(t_a - t)/(t_k - t_0) \asymp 1$, showing that $\nu$ cannot be arbitrarily large.

    Now we check that
    \[ \rho := \frac{h_n(t + \delta) - h_n(t)}{h_n(t) - h_n(t - \delta)} \gtrsim a_m. \]
    On one side, we have
    \[ h_n(t + \delta) - h_n(t) = t'_{\nu + 1} - h_n(t) + \sum_{j=\nu + 2}^{a} t'_j - t'_{j-1} \gtrsim \sum_{j=\nu + 2}^{a} \frac{t'_k - t'_0}{k} = \frac{a - \nu - 2}{k}(t'_k - t'_0). \]
    On the other,
    \[ h_n(t) - h_n(t - \delta) = h_n(t) - t'_\nu + \sum_{j=1}^{\nu} t'_j - t'_{j-1} \lesssim \sum_{j=1}^{\nu + 1} \frac{t'_k - t'_0}{k} = \frac{\nu + 1}{k}(t'_k - t'_0). \]
    Thus $\rho \gtrsim (a - \nu - 2)/(\nu + 1) \asymp a_m$.
\end{proof}

With such a control of the quasisymmetry constant, we can now see that a couple of usual extensions would generally result in worst distortions. For instance, if we denote by $E_{BA}$ the Beurling-Ahlfors extension, then Lehtinen showed in \cite{Lehtinen1983} that
\[ K_{E_{BA}}(g) \leq 2\| g \|_{qs} \]
for any quasisymmetric map $g: \R \to \R$. Thus, this bound would only get us a linear control of the distortion, instead of the logarithmic one that the Yoccoz extension provides. The Douady-Earle extension, here denoted $E_{DE}$, is much worse. Douady and Earle \cite{Douady-Earle1986}, and  Partyka \cite{Partyka1994} showed that
\[ K_{E_{DE}}(g) \leq G\left( \min\left\{ \| g \|_{qs}^{3/2}, 2\| g \|_{qs} - 1 \right\} \right) \]
for some explicit growth function $G: [1, \infty) \to [1, \infty)$ that grows asymptotically faster than polynomial. These are not concrete proofs of optimality, since these bounds are very general and might not be sharp in our context, but they illustrate that, by taking the underlying dynamics into consideration, the Yoccoz extension can massively improve traditional distortion controls.

\bibliographystyle{plain}
\bibliography{biblio}

\end{document}